\documentclass[reqno]{amsart}
\usepackage{amsmath,amsthm,amscd,amssymb,amsfonts, amsbsy}
\usepackage{latexsym, color, enumerate}
\usepackage{pxfonts}

\theoremstyle{plain}
\newtheorem{theorem}[equation]{Theorem}
\newtheorem{lemma}[equation]{Lemma}
\newtheorem{corollary}[equation]{Corollary}

\theoremstyle{definition}
\newtheorem{definition}[equation]{Definition}

\theoremstyle{remark}
\newtheorem{remark}[equation]{Remark}

\newcommand{\dv}{\operatorname{div}}

\newcommand{\tr}{\operatorname{tr}}

\numberwithin{equation}{section}

\newcommand{\bR}{\mathbb{R}}

\newcommand\cL{\mathcal{L}}

\providecommand{\set}[1]{\{#1\}}

\providecommand{\abs}[1]{\lvert#1\rvert}
\providecommand{\Abs}[1]{\left\lvert#1\right\rvert}

\providecommand{\norm}[1]{\lVert#1\rVert}

\renewcommand{\vec}[1]{\boldsymbol{#1}}

\begin{document}
\title[Estimates for linear elliptic operators]
{On $C^1$, $C^2$, and weak type-$(1,1)$ estimates for linear elliptic operators}

\author[H. Dong]{Hongjie Dong}
\address[H. Dong]{Division of Applied Mathematics, Brown University,
182 George Street, Providence, RI 02912, United States of America}
\email{Hongjie\_Dong@brown.edu}
\thanks{H. Dong was partially supported by the NSF under agreement DMS-1056737 and DMS-1600593.}

\author[S. Kim]{Seick Kim}
\address[S. Kim]{Department of Mathematics, Yonsei University, 50 Yonsei-ro, Seodaemun-gu, Seoul 03722, Republic of Korea}
\email{kimseick@yonsei.ac.kr}
\thanks{S. Kim is partially supported by NRF Grant No. NRF-20151009350.}



\begin{abstract}
We show that any weak solution to elliptic equations in divergence form is continuously differentiable provided that the modulus of continuity of coefficients in the $L^1$-mean sense satisfies the Dini condition.
This in particular answers a question recently raised by Yanyan Li \cite{Y.Li2016} and allows us to improve a result of Brezis \cite{Br08}.
We also prove a weak type-$(1,1)$ estimate under a stronger assumption on the modulus of continuity.
The corresponding results for non-divergence form equations are also established.
\end{abstract}
\maketitle

\section{Introduction and main results}
Let $L$ be a second-order elliptic operator in divergence form
\begin{equation*}	
L u= \dv (A(x) \nabla u)=\sum_{i,j=1}^d D_i(a^{ij}(x) D_j u).
\end{equation*}
Here, we assume that the coefficients $A=(a^{ij})_{i,j=1}^d$ are defined on a domain $\Omega \subset \bR^d$ and satisfy the ellipticity condition
\[
\lambda \abs{\xi}^2 \le \textstyle \sum_{i,j=1}^d a^{ij}(x) \xi_i \xi_j,\quad \forall \xi=(\xi_1,\ldots, \xi_d)^\top \in \bR^d,\quad\forall x \in \Omega
\]
and boundedness condition
\[
\textstyle\sum_{i,j=1}^d \,\abs{a^{ij}(x)}^2 \le \Lambda^2,\quad \forall x \in \Omega
\]
for some constants $\lambda, \Lambda>0$.
It is well known that any weak solution to $Lu=\dv \vec g$ is continuously differentiable provided that $A$ and $\vec g$ satisfies the $\alpha$-increasing Dini continuity condition for some $\alpha\in (0,1]$; i.e., the moduli of continuity of $A$ and $\vec g$ in the $L^\infty$ sense is bounded by a continuous, increasing function $\varphi:[0,\infty) \to [0,\infty)$, which satisfies
$\varphi(0)=0$,
\begin{equation}			\label{eq-dini}
\int_0^1 \frac{\varphi(r)}{r}\,dr <+\infty,
\end{equation}
and $\varphi(s)/s^\alpha$ is a decreasing function.
See, for instance, \cite{ME65,Bu78,Lieb87,MM11}.

In a recent paper \cite{Y.Li2016}, Yanyan Li raised a question regarding $C^1$ regularity of weak solution of  $Lu= 0$ in $B_4=B(0,4)$ if the Dini condition~\eqref{eq-dini} holds for
\[
\varphi(r):=\sup_{x\in B_3} \left( \fint_{B(x,r)} \abs{A-\bar A_{B(x,r)}}^2 \right)^{\frac12},
\;\text{ where }\;
\bar A_{B(x,r)}=\fint_{B(x,r)} A:=\frac{1}{\abs{B(x,r)}} \int_{B(x,r)} A.
\]
We found that the answer is positive.
In fact, we can replace $L^2$ average in $\varphi(r)$ by $L^1$ average; i.e., we have $C^1$ regularity if we assume
\begin{equation}			\label{eq:dcmo}
\int_0^1 \frac{ \omega_A(r)}{r}\,dr <+\infty,\;\mbox{ where }\;
\omega_A(r):=\sup_{x\in B_3} \fint_{B(x,r)} \abs{A-\bar A_{B(x,r)}}.
\end{equation}

\begin{definition}
                    \label{def1.3}
For a function $g$ defined on $B_4=B(0,4)$, we shall say that $g$ has \emph{Dini mean oscillation} if it satisfies the following condition:
\begin{equation}					\label{13.24f}
\int_0^1 \frac{ \omega_g(r)}{r}\,dr <+\infty,\;\mbox{ where }\;
\omega_g(r):=\sup_{x\in B_3} \fint_{B(x,r)} \abs{g-\bar {g}_{B(x,r)}}.
\end{equation}
\end{definition}
The Dini mean oscillation condition is weaker than the usual Dini continuity condition mentioned above.
For example, if $a^{ij}(0)=\delta_{ij}$ and for $0<\abs{x} \ll 1$,
\[
a^{ij}(x)=\delta_{ij}\left(1+(-\ln \abs{x})^{-\gamma}\right),
\]
where $0<\gamma <\frac12$, then $A$ is neither Dini continuous nor satisfies the square Dini condition considered in \cite{MM11}.
However, a simple calculation reveals that
\[
\omega_A(r)\sim  (-\ln r)^{-\gamma-1},
\]
which implies that $A$ satisfies the Dini mean oscillation condition in Definition \ref{def1.3}.
We formulate our result more precisely as the following theorem.

\begin{theorem}					\label{thm-main-d}
Suppose the coefficients $A(x)$ of the divergent operator $L$ have Dini mean oscillation; i.e., $A$ satisfies \eqref{eq:dcmo}.
Let $u \in W^{1,2}(B_4)$ be a weak solution of
\[
Lu= \dv (A(x) \nabla u)=\dv \vec g\;\mbox{ in } \; B_4=B(0,4),
\]
where $\vec g$ has Dini mean oscillation; i.e., $\vec g$ satisfies \eqref{13.24f}.
Then, we have $u\in C^1(\overline{B_1})$.
\end{theorem}
An upper bound of the modulus of continuity of $Du$ can be found in the proof in Section~\ref{sec2}.
We also consider an elliptic operator $\cL$ in non-divergence form
\begin{equation*}	
\cL u= \tr (A(x) D^2 u)=\sum_{i,j=1}^d a^{ij}(x) D_{ij} u,
\end{equation*}
where the coefficients $A(x)$ are assumed to be symmetric and satisfy the same ellipticity and boundedness condition as above.

\begin{theorem}					\label{thm-main-nd}
Suppose the coefficients $A(x)$ of the non-divergent operator $\cL$ have Dini mean oscillation; i.e., $A$ satisfies  \eqref{eq:dcmo}.
Let $u \in W^{2,2}(B_4)$ be a strong solution of
\[
\cL u= \tr (A(x) D^2 u)= g\;\mbox{ in } \; B_4=B(0,4),
\]
where $g$ has Dini mean oscillation; i.e., $g$ satisfies the condition \eqref{13.24f}.
Then, we have $u\in C^2(\overline{B_1})$.
\end{theorem}

In proving the above theorem, we need to consider the formal adjoint operator defined by
\begin{equation*}	
\cL^* u = \sum_{i,j=1}^d D_{ij}(a^{ij}(x) u)
\end{equation*}
and deal with boundary value problems of the form
\begin{equation}					\label{16.25f}
\cL^* u =\dv^2 \mathbf{g}\;\mbox{ in }\;\Omega;\quad u=\frac{\mathbf{g} \nu \cdot \nu}{A \nu \cdot  \nu}\;\mbox{ on }\;\partial \Omega,
\end{equation}
where $\mathbf{g}=(g^{kl})_{k, l=1}^d$, $\dv^2 \mathbf{g}=\sum_{k, l=1}^d D_{kl} g^{kl}$, with $\mathbf{g} \in L^p(\Omega)$.
The following definition is extracted from Escauriaza and Montaner~\cite{EM2016}.
\begin{definition}			\label{def-adj}
Let $\Omega \subset \bR^d$ be a bounded $C^{1,1}$ domain with unit exterior normal vector $\nu$.
Let $\mathbf{g} \in L^p(\Omega)$, $1<p<\infty$ and $\frac1p+\frac1{p'}=1$.
We say that $u \in L^p(\Omega)$ is an adjoint solution of \eqref{16.25f} if $u$ satisfies
\begin{equation}				\label{16.33f}
\int_\Omega u\, \cL v = \int_\Omega \tr(\mathbf{g}\, D^2 v),
\end{equation}
for any $v$ in $W^{2,p'}(\Omega) \cap W^{1,p'}_0(\Omega)$.
By a local adjoint solution of
\[
\cL^* u =\dv^2 \mathbf{g}\;\mbox{ in }\;\Omega,
\]
we mean a solution in $L^p_{loc}(\Omega)$ that verify \eqref{16.33f} when $v$ is in $W^{2,p'}_0(\Omega)$.
\end{definition}

\begin{theorem}					\label{thm-main-adj}
Suppose the coefficients $A(x)$ of the non-divergent operator $\cL$ have Dini mean oscillation; i.e., $A$ satisfies  \eqref{eq:dcmo}.
Let $u \in L^2(B_4)$ is a local adjoint solution of
\[
\cL^* u= \sum_{i,j=1}^d D_{ij} (a^{ij}(x) u)= \dv^2 \mathbf g \;\mbox{ in } \; B_4=B(0,4),
\]
where $\mathbf g=(g^{kl})_{k, l=1}^d$ have Dini mean oscillation; i.e., $\mathbf g$ satisfy the condition \eqref{13.24f}.
Then, we have $u\in C(\overline{B_1})$.
\end{theorem}

The proofs of Theorems \ref{thm-main-d}, \ref{thm-main-nd}, and \ref{thm-main-adj} are based on Campanato's approach, which was used previously, for instance, in \cite{Giaq83, Lieb87}. The main step of Campanato's approach is to show the mean oscillations of $Du$ (or $D^2u$, or $u$, respectively) in balls vanish in certain order as the radii of balls go to zero.
The main difficulty is that because we only impose the assumption on the $L^1$-mean oscillation of $A$ and $g$, the usual argument based on the $L^2$ (or $L^p$ for $p>1$) estimates does not work in our case.
To this end, we exploit weak type-$(1,1)$ estimates, the proof of which use a duality argument.
We then adapt Campanato's idea in the $L^p$ setting for some $p\in (0,1)$.

\begin{remark}
Theorems \ref{thm-main-d}, \ref{thm-main-nd}, and \ref{thm-main-adj} remain true for corresponding elliptic systems satisfying the Legendre-Hadamard ellipticity condition:
\[
\lambda \abs{\xi}^{2} \,\abs{\eta}^2 \le
\sum_{\alpha,\beta=1}^n\sum_{i,j=1}^d
a^{\alpha\beta}_{ij}(x)\xi_{i}\xi_{j} \eta_\alpha \eta_\beta
\]
for all $x \in \Omega$, $\xi \in \bR^d$, and $\eta \in \bR^n$.
This is because the proofs below use the $W^{1,p}$ (or $W^{2,p})$ solvability for elliptic equations in balls with constant coefficients, which is also available for elliptic systems which satisfy the Legendre-Hadamard ellipticity condition.
See, for instance, \cite{DK11} and the references therein.
\end{remark}

The remaining part of the paper is organized as follows.
In Section~\ref{sec2}, we prove our main results, namely, Theorems \ref{thm-main-d}, \ref{thm-main-nd}, and \ref{thm-main-adj}.
Section~\ref{sec3} is devoted to the weak type-$(1,1)$ estimates under an additional  condition on the mean oscillation $\omega_A$.
Finally in Section~\ref{sec4}, we give an application of our main theorems, which improves a result of Brezis~\cite{Br08} as well as a more recent result by Escauriaza and Montaner~\cite{EM2016}.

\section{Proof of main theorems}				\label{sec2}
To begin with, let us introduce some notation that will be used throughout the proof.
For nonnegative (variable) quantities $A$ and $B$, the relation $A \lesssim B$ should be understood that there is some constant $c>0$ such that $A \le c B$.
We write $A \sim B$ if $A\lesssim B$ and $B \lesssim A$.

\subsection{Preliminary lemmas}
\begin{lemma}			\label{lem01-stein}
Let $B=B(0,1)$ and let $T$ be a bounded linear operator from $L^2(B)$ to $L^2(B)$.
Suppose that for any $\bar y \in B$ and $0<r<\frac12$ we have
\[
\int_{B \setminus B(\bar y, cr)} \abs{Tb} \le C \int_{B(\bar y, r)\cap B} \abs{b}
\]
whenever $b \in L^2(B)$ is supported in $B(\bar y, r)\cap B$, $\int_{B} b =0$, and $c>1$ and $C>0$ are constants.
Then for $f \in L^2(B)$ and any $\alpha>0$, we have
\[
\Abs{\set{x \in B : \abs{Tf(x)} > \alpha}}  \le \frac{C'}{\alpha} \int_{B} \abs{f},
\]
where $C'=C'(d,c,C)$ is a constant.
\end{lemma}
\begin{proof}
We refer to Stein \cite[p. 22]{Stein93}, where the proof is based on the Calder\'on-Zygmund decomposition and the domain is assumed to be the whole space.
In our case, we can modify the proof there by using the ``dyadic cubes'' decomposition of $B$; see Christ \cite{Ch90}.
\end{proof}

The following lemma (as well as subsequent lemmas \ref{lem-weak11-nd} and \ref{lem-weak11-adj}) should be classical.
Here, we provide a proof that does not explicitly involve singular integrals, which will be useful for our discussion in Section~\ref{sec3}.

\begin{lemma}			\label{lem02-weak11}
Let $B=B(0,1)$.
Let $L_0$ be an operator with constant coefficients.
For $\vec f \in L^2(B)$ let $u \in  W^{1,2}_0(B)$ be a unique weak solution to
\[
L_0 u= \dv \vec f\;\mbox{ in }\; B;\quad
u=0 \;\mbox{ on } \; \partial B.
\]
Then for any $\alpha>0$, we have
\[
\Abs{\set{x \in B : \abs{Du(x)} > \alpha}}  \le \frac{C}{\alpha} \int_{B} \abs{\vec f}\,dx,
\]
where $C=C(d, \lambda, \Lambda)$.
\end{lemma}
\begin{proof}
Note that the map $T: \vec f \mapsto Du$ is a bounded linear operator on $L^2(B)$.
It is enough to show that $T$ satisfies the hypothesis of Lemma~\ref{lem01-stein}.
We set $c=2$.
Fix $\bar y \in B$ and $0<r<\frac12$, let $\vec b \in L^2(B)$ be supported in $B(\bar y, r)\cap B$ with mean zero.
Let $u \in W^{1,2}_0(B)$ be the unique weak solution of
\[
L_0 u= \dv \vec b\;\mbox{ in }\; B;\quad
u=0 \;\mbox{ on } \; \partial B.
\]
For any $R\ge cr$ such that $B\setminus B(\bar y, R) \neq \emptyset$ and $\vec g \in C^\infty_c((B(\bar y, 2R)\setminus B(\bar y,R))\cap B)$,
let  $v \in W^{1,2}_0(B)$ be a weak solution of
\[
L_0^* v= \dv \vec g\;\mbox{ in }\; B;\quad
v=0 \;\mbox{ on } \; \partial B,
\]
where $L_0^*$ is the adjoint operator of $L_0$.
By the definition of weak solutions, we have the identity
\[
\int_{B} Du \cdot \vec g = \int_{B} \vec b \cdot Dv = \int_{B(\bar y ,r) \cap B} \vec b \cdot \left(Dv-\overline{Dv}_{B(\bar y ,r) \cap B} \right).	
\]
Therefore,
\begin{align}
						\label{14.23m}
\Abs{\int_{(B(\bar y, 2R)\setminus B(\bar y,R))\cap B} Du \cdot \vec g\,} &\le \norm{\vec b}_{L^1(B(\bar y, r) \cap B)} \norm{Dv-\overline{Dv}_{B(\bar y ,r) \cap B}}_{L^\infty(B(\bar y, r)\cap B)}\\
						\label{14.18m}
&\lesssim r \norm{\vec b}_{L^1(B(\bar y, r) \cap B)} \norm{D^2v}_{L^\infty(B(\bar y, r)\cap B)}.
\end{align}
Since $L_0^* v =0$ in $B(\bar y, R)\cap B$ and $r\le R/c=R/2$, by interior and boundary estimates for elliptic systems with constant coefficients, we get
\begin{align*}
\norm{D^2v}_{L^\infty(B(\bar y, r)\cap B)}
&\lesssim R^{-1-\frac{d}2} \norm{Dv}_{L^2(B(\bar y, R)\cap B)} \lesssim R^{-1-\frac{d}2} \norm{Dv}_{L^2(B)} \\
&\lesssim R^{-1-\frac{d}2} \norm{\vec g}_{L^2(B)}= R^{-1-\frac{d}2} \norm{\vec g}_{L^2((B(\bar y, 2R)\setminus B(\bar y,R))\cap B)}.
\end{align*}
Therefore, by the duality, we get
\[
\norm{Du}_{L^2((B(\bar y, 2R)\setminus B(\bar y,R))\cap B)}
\lesssim r R^{-1-\frac{d}2} \norm{\vec b}_{L^1(B(\bar y, r)\cap B)},
\]
and thus, we have
\begin{equation}
                            \label{eq4.37}
\norm{Du}_{L^1((B(\bar y, 2R)\setminus B(\bar y,R))\cap B)}
\lesssim r R^{-1}\norm{\vec b}_{L^1(B(\bar y, r) \cap B)}.
\end{equation}
Now let $N$ be the smallest positive integer such that $B\subset B(\bar y, 2^{N}cr)$.
By taking $R=cr,2cr,\ldots,2^{N-1} cr$ in \eqref{eq4.37}, we have
\begin{equation}			\label{14.06m}
\int_{B \setminus B(\bar y, cr)} \abs{D u} \lesssim \sum_{k=1}^{N} 2^{-k} \norm{\vec b}_{L^1(B(\bar y, r)\cap B)} \sim \int_{B(\bar y, r)\cap B} \abs{\vec b}.
\end{equation}
Therefore, $T$ satisfies the hypothesis of Lemma~\ref{lem01-stein}.
The lemma is proved.
\end{proof}

\begin{lemma}			\label{lem03-dini}
Let $\omega$ be a nonnegative bounded function.
Suppose there is $c_1, c_2 >0$ and $0<\kappa<1$ such that
\begin{equation}					\label{eq14.59sa}
c_1 \omega(t) \le \omega(s) \le c_2 \omega(t)\quad\text{whenever }\;\kappa t \le s \le  t\;\mbox{ and }\; 0<t<r.
\end{equation}
Then, we have
\[
\sum_{i=0}^\infty \omega(\kappa^i r) \lesssim \int_0^r \frac{\omega(t)}{t}\,dt.
\]
\end{lemma}
\begin{proof}
Immediate from the comparison principle for Riemann integrals.
\end{proof}

\subsection{Proof of Theorem~\ref{thm-main-d}}
We shall derive an a priori estimate of the modulus of continuity of $Du$ by assuming that $u$ is in $C^1(B_3)$. The general case follows from a standard approximation argument; see, for instance, \cite[p. 134]{Dong2012}.
For $\bar x \in B_3$ and $0<r<\frac13$, we consider the quantity
\[
\phi(\bar x,r):=\inf_{\vec q\in \bR^d}\left( \fint_{B(\bar x,r)}
\abs{Du - \vec q}^p \right)^{\frac1p},
\]
where $0<p<1$ is some fixed exponent.
First of all, we note that
\begin{equation}				\label{eq12.14f}
\phi(\bar x, r) \le \left(\fint_{B(\bar x,r)} \abs{Du}^p\right)^{\frac1p} \lesssim r^{-d} \norm{Du}_{L^1(B(\bar x,r))}.
\end{equation}
We want to control the quantity $\phi(\bar x, r)$.
To this end, we decompose $u=v+w$, where $w \in W^{1,2}_0(B(\bar x, r))$ is the weak solution of the problem
\[
L_0 w = -\dv ((A- \bar A) \nabla u) +\dv (\vec g-\bar{\vec g})\;\mbox{ in }\;B(\bar x, r);\quad
w=0 \;\mbox{ on }\;\partial B(\bar x, r).
\]
Here and below, we use the simplified notation
\[
\bar A=\bar A_{B(\bar x, r)},\quad \bar{\vec g} = \bar{\vec g}_{B(\bar x, r)},
\]
and $L_0$ is the elliptic operator with the constant coefficients $\bar A$.
By Lemma~\ref{lem02-weak11} with scaling, we have
\begin{equation}
                                    \label{eq10.20}
\abs{\set{x\in B(\bar x, r): \abs{Dw(x)} > \alpha}}
\lesssim \frac{1}{\alpha}\left(\norm{Du}_{L^\infty(B(\bar x, r))} \int_{B(\bar x, r)} \abs{A-\bar A} +  \int_{B(\bar x, r)}  \abs{\vec g -\bar {\vec g}} \,\right).
\end{equation}
For any given $0<p<1$, recall the formula
\begin{align*}
\int_{B(\bar x, r)} \abs{Dw}^p
&=\int_0^\infty p\alpha^{p-1}\abs{\set{x\in B(\bar x, r): \abs{Dw(x)} > \alpha}}\,d\alpha\\
&=\int_0^\tau + \int_\tau^\infty p\alpha^{p-1}\abs{\set{x\in B(\bar x, r): \abs{Dw(x)} > \alpha}}\,d\alpha,
\end{align*}
where $\tau>0$ is to be determined in a moment.
When $0< \alpha \le \tau$, we bound $\abs{\set{x\in B(\bar x, r): \abs{Dw(x)} > \alpha}}$ simply by $\abs{B(\bar x, r)}= C r^d$.
When $\alpha > \tau$, we use \eqref{eq10.20}. It then follows that
\[
\int_{B(\bar x, r)} \abs{Dw}^p\le \tau^p\abs{B(\bar x, r)}+\frac{p}{1-p}\,\tau^{p-1}\left(\norm{Du}_{L^\infty(B(\bar x, r))} \int_{B(\bar x, r)} \abs{A-\bar A} +  \int_{B(\bar x, r)}  \abs{\vec g -\bar {\vec g}} \,\right).
\]
By optimizing over $\tau$,
we get
\[
\int_{B(\bar x, r)} \abs{Dw}^p\le \frac{1}{1-p}\,\abs{B(\bar x, r)}^{1-p} \left(\norm{Du}_{L^\infty(B(\bar x, r))} \int_{B(\bar x, r)} \abs{A-\bar A} +  \int_{B(\bar x, r)}  \abs{\vec g -\bar {\vec g}} \,\right)^p.
\]
Therefore, we have
\begin{equation}				\label{eq15.50a}
\left(\fint_{B(\bar x, r)} \abs{Dw}^p \right)^{\frac1p} \lesssim \omega_A(r) \,\norm{Du}_{L^\infty(B(\bar x, r))} +  \omega_g(r).
\end{equation}
On the other hand, note that $v=u-w$ satisfies
\[
L_0 v =0 \;\mbox{ in }\;B(\bar x, r),
\]
and that for any constant vector $\vec q\in \bR^d$, the same equation is satisfied by $Dv - \vec q$.
By the interior estimates for elliptic systems with constant coefficients, we have
\begin{equation}			\label{int-reg}
\norm{D^2 v}_{L^\infty(B(\bar x, \frac12 r))} \le C_0 r^{-1} \left(\fint_{B(\bar x, r)} \abs{Dv - \vec q}^p\,\right)^{\frac1p},
\end{equation}
where $C_0>0$ is an absolute constant.
Let $0< \kappa < \frac12$ to be a number to be fixed later.
Then, we have
\begin{equation}				\label{eq15.50b}
\left(\fint_{B(\bar x, \kappa r)} \abs{Dv - \overline{Dv}_{B(\bar x, \kappa r)}}^p \right)^{\frac1p} \le 2\kappa r \norm{D^2 v}_{L^\infty(B(\bar x, \frac12 r))} \le 2C_0 \kappa \left(\fint_{B(\bar x, r)} \abs{Dv - \vec q}^p \, \right)^{\frac1p}.
\end{equation}
Here, we recall the facts that for $0<p<1$, we have for all $a, b \ge 0$ that
\[
(a+b)^p \le a^p + b^p, \quad (a^p+ b^p)^{\frac1p} \le 2^{\frac1p} (a+b),
\]
and for any measure space $(X, \mu)$, we have
\[
\norm{f+g}_{L^p(X,\mu)} \le 2^{(1-p)/p} \left( \norm{f}_{L^p(X, \mu)} + \norm{g}_{L^p(X, \mu)}\right).
\]
By using the decomposition $u=v+w$, we obtain from \eqref{eq15.50b} that
\begin{align*}
&\left(\fint_{B(\bar x, \kappa r)} \abs{Du - \overline{Dv}_{B(\bar x, \kappa r)}}^p \right)^{\frac1p} \\
&\qquad\qquad \le 2^{\frac{1-p}p} \left(\fint_{B(\bar x, \kappa r)} \abs{Dv - \overline{Dv}_{B(\bar x, \kappa r)}}^p \right)^{\frac1p}+ C\left(\fint_{B(\bar x, \kappa r)} \abs{Dw}^p \right)^{\frac1p}\\
&\qquad \qquad \le  2^{\frac2p-1}C_0\kappa \left(\fint_{B(\bar x, r)} \abs{Du - \vec q}^p \right)^{\frac1p} +C (\kappa^{-\frac{d}{p}}+1) \left(\fint_{B(\bar x, r)} \abs{Dw}^p \right)^{\frac1p}.
\end{align*}
Since $\vec q\in \bR^d$ is arbitrary, by using \eqref{eq15.50a}, we thus obtain
\[
\phi(\bar x,\kappa r) \le 2^{\frac2p-1}C_0 \kappa \, \phi(\bar x, r)+ C (\kappa^{-\frac{d}{p}}+1) \left( \omega_A(r) \,\norm{Du}_{L^\infty(B(\bar x, r))} +  \omega_g(r) \right).
\]
Now we choose $\kappa$ sufficiently small so that $2^{\frac2p}C_0\kappa \le 1$.
Then, we obtain
\[
\phi(\bar x, \kappa r) \le \frac12 \phi(\bar x, r)+ C \left( \omega_A(r) \,\norm{Du}_{L^\infty(B(\bar x, r))} +  \omega_g(r) \right).
\]
By iterating, for $j=1,2,\ldots$, we get
\[
\phi(\bar x, \kappa^j r) \le 2^{-j} \phi(\bar x, r) +C \norm{Du}_{L^\infty(B(\bar x,r))} \sum_{i=1}^{j} 2^{1-i} \omega_A(\kappa^{j-i} r) + C \sum_{i=1}^{j} 2^{1-i} \omega_g(\kappa^{j-i} r).
\]
Therefore, we have
\begin{equation}				\label{eq22.25f}
\phi(\bar x, \kappa^j r) \le 2^{-j} \phi(\bar x, r) +C \norm{Du}_{L^\infty(B(\bar x, r))}\,\tilde \omega_A(\kappa^{j} r) + C \tilde \omega_g(\kappa^{j} r),
\end{equation}
where we set
\begin{equation}				\label{eq10.54}
\tilde \omega_{\bullet}(t)=\sum_{i=1}^{\infty} \frac{1}{2^{i}} \left(\omega_{\bullet}(\kappa^{-i}t)\, [\kappa^{-i}t \le 1] + \omega_{\bullet}(1)\, [\kappa^{-i}t >1] \right).
\end{equation}
Here, we used Iverson bracket notation; i.e., $[P]=1$ if $P$ is true and $[P]=0$ otherwise.
We remark that $\int_0^1 \tilde \omega_{\bullet}(t) / t \,dt <\infty$; see \cite[Lemma~1]{Dong2012}.

Now, we take $\vec q_{\bar x,r}\in \bR^d$ to be such that
\[
\phi(\bar x,r)=\left( \fint_{B(\bar x,r)} \abs{Du - \vec q_{\bar x,r}}^p\, \right)^{\frac1p}.
\]
Similarly, we find $\vec q_{\bar x,\kappa r} \in \bR^d$, et cetera.
Since we have
\[
\abs{\vec q_{\bar x,\kappa r} - \vec q_{\bar x, r}}^p \le
\abs{Du(x)-\vec q_{\bar x,r}}^p + \abs{Du(x) - \vec q_{\bar x,\kappa r}}^p,
\]
by taking average over $x \in B(\bar x,\kappa r)$ and then taking $p$th root, we obtain
\[
\abs{\vec q_{\bar x,\kappa r} -\vec q_{\bar x,r}} \lesssim \phi(\bar x,\kappa r) + \phi(\bar x, r).
\]
Then, by iterating, we get
\[
\abs{\vec q_{\bar x,\kappa^i r} - \vec q_{\bar x, r}} \lesssim \sum_{j=0}^i \phi(\bar x, \kappa^j r).
\]
Since the right-hand side of \eqref{eq22.25f} goes to zero as $j\to \infty$, by the assumption that $u \in C^1(B_3)$, we find
\[
\lim_{i\to \infty} \vec q_{\bar x,\kappa^i r}=Du(\bar x).
\]
Therefore, by taking $i\to \infty$, using \eqref{eq22.25f} and Lemma~\ref{lem03-dini}, we get
\begin{equation}			\label{eq13.13}
\abs{Du(\bar x)-\vec q_{\bar x,r}} \lesssim \phi(\bar x, r)+
\norm{Du}_{L^\infty(B(\bar x,r))} \int_0^r \frac{\tilde\omega_A(t)}t \,dt+\int_0^r \frac{\tilde\omega_g(t)}t \,dt.
\end{equation}
Indeed, it is easy to see that $\omega_A$ and $\omega_g$ satisfy \eqref{eq14.59sa}; see, for instance, \cite[p. 7]{Y.Li2016}.
By the definition \eqref{eq10.54}, so do $\tilde \omega_A$ and $\tilde \omega_g$.
By averaging the inequality
\[
\abs{\vec q_{\bar x, r}}^p \le \abs{Du(x) -\vec q_{\bar x, r}}^p + \abs{Du(x)}^p
\]
over $x \in B(\bar x, r)$ and taking $p$th root, we get
\[
\abs{\vec q_{\bar x, r}} \le 2^{\frac1p} \phi(\bar x, r) + 2^{\frac1p} \left(\fint_{B(\bar x,r)} \abs{Du}^p\right)^{\frac1p}.
\]
Therefore, by \eqref{eq13.13} and \eqref{eq12.14f}, we get
\[
\abs{Du(\bar x)}  \lesssim r^{-d} \norm{Du}_{L^1(B(\bar x,r))} + \norm{Du}_{L^\infty(B(\bar x,r))} \int_0^r \frac{\tilde\omega_A(t)}t \,dt+\int_0^r \frac{\tilde\omega_g(t)}t \,dt.
\]
Now, taking supremum for $\bar x\in B(x_0, r)$, where $x_0 \in B_2$ and $r<\frac 13$, we have
\[
\norm{Du}_{L^\infty(B(x_0, r))} \le
C \left( r^{-d} \norm{Du}_{L^1(B(x_0, 2r))} + \norm{Du}_{L^\infty(B(x_0, 2r))} \int_0^r \frac{\tilde\omega_A(t)}t \,dt +\int_0^r \frac{\tilde\omega_g(t)}t \,dt \right).
\]
We fix $r_0<\frac13$ such that for any $0<r\le r_0$,
\[
C \int_0^{r} \frac{\tilde\omega_A(t)}t \,dt \le \frac1{3^d}.
\]
Then, we have for any $x_0 \in B_2$ and $0<r\le r_0$ that
\[
\norm{Du}_{L^\infty(B(x_0, r))} \le
3^{-d}\norm{Du}_{L^\infty(B(x_0,2r))} + C r^{-d} \norm{Du}_{L^1(B(x_0,2r))} + C \int_0^r \frac{\tilde\omega_g(t)}t \,dt.
\]
For $k=1,2,\ldots$, denote $r_k=3-2^{1-k}$.
Note that $r_{k+1}-r_k=2^{-k}$ for $k\ge 1$ and $r_1=2$.
For $x_0\in B_{r_k}$ and $r\le 2^{-k-2}$, we have $B(x_0, 2r) \subset B_{r_{k+1}}$. We take $k_0\ge 1$ sufficiently large such that $2^{-k_0-2}\le r_0$.
It then follows that for any $k\ge k_0$,
\[
\norm{Du}_{L^\infty(B_{r_k})} \le 3^{-d} \norm{Du}_{L^\infty(B_{r_{k+1}})}+C 2^{kd} \norm{Du}_{L^1(B_4)} + C \int_0^{1} \frac{\tilde\omega_g(t)}t \,dt.
\]
By multiplying the above by $3^{-kd}$ and then summing over $k=k_0, k_0+1,\ldots$, we reach
\[
\sum_{k=k_0}^\infty 3^{-kd}\norm{Du}_{L^\infty(B_{r_k})} \le \sum_{k=k_0}^\infty 3^{-(k+1)d} \norm{Du}_{L^\infty(B_{r_{k+1}})}+C \norm{Du}_{L^1(B_4)} + C \int_0^{1} \frac{\tilde\omega_g(t)}t \,dt.
\]
Since we assume that $u\in C^1(B_3)$, the summations on both sides are convergent.
Therefore, we have
\begin{equation}					\label{eq10.23m}
\norm{Du}_{L^\infty(B_2)} \lesssim  \norm{Du}_{L^1(B_4)} + \int_0^{1} \frac{\tilde\omega_g(t)}t \,dt.
\end{equation}
Next, note that by \eqref{eq13.13}, we have for $0<r<\frac 13$ that
\[
\sup_{\bar x \in B_1}\, \abs{Du(\bar x) - \vec q_{\bar x ,r}} \lesssim \psi(r)
:=  \sup_{\bar x \in B_1} \phi(\bar x, r)+\norm{Du}_{L^\infty(B_2)} \int_0^r \frac{\tilde\omega_A(t)}t \,dt+\int_0^r \frac{\tilde\omega_g(t)}t \,dt.
\]
By \eqref{eq22.25f} and \eqref{eq12.14f}, we find
\[
\sup_{\bar x \in B_1} \phi(\bar x, r) \lesssim r^\beta \norm{Du}_{L^1(B_2)}+
\tilde\omega_A(r)\, \norm{Du}_{L^\infty(B_2)} + \tilde \omega_g(r),\;\text{ where }\;\beta:=\tfrac{\ln \frac12}{\ln \kappa}>0.
\]
On the other hand, for $x, y \in B_1$ such that $\abs{x-y} <\frac12$, we have
\begin{align*}
\abs{Du(x)-Du(y)}^p &\le \abs{Du(x)- \vec q_{x,r}}^p + \abs{\vec q_{x,r}-\vec q_{y,r}}^p+\abs{Du(y)-\vec q_{y,r}}^p\\
&\le 2 \psi^p(r)+ \abs{Du(z)-\vec q_{x,r}}^p+\abs{Du(z)-\vec q_{y,r}}^p.
\end{align*}
We set $r=\abs{x-y}$, take the average over $z \in B(x,r)\cap B(y,r)$, and then take the $p$th root to get
\begin{equation}					\label{eq13.27m}
\abs{Du(x)-Du(y)} \lesssim \psi(r)+ \phi(x,r) + \phi(y,r) \lesssim \psi(r).
\end{equation}
Therefore, we get from \eqref{eq13.27m} and \eqref{eq10.23m} that
\begin{multline}				\label{eq14.00}
\abs{Du(x)-Du(y)} \lesssim \norm{Du}_{L^1(B_4)} \,\abs{x-y}^\beta\\
+\left(\norm{Du}_{L^1(B_4)} + \int_0^{1} \frac{\tilde\omega_g(t)}t \,dt\right) \int_0^{\abs{x-y}} \frac{\tilde \omega_A(t)}t \,dt + \int_0^{\abs{x-y}} \frac{\tilde \omega_g(t)}t \,dt,
\end{multline}
where we also used the fact that $\tilde \omega_{\bullet}$ satisfies the hypothesis of Lemma~\ref{lem03-dini}.
The theorem is proved.

\subsection{Proof of Theorem~\ref{thm-main-nd}}
The proof of theorem is parallel to that of Theorem~\ref{thm-main-d}.
First, we present a lemma that plays the role of Lemma~\ref{lem02-weak11}.
\begin{lemma}			\label{lem-weak11-nd}
Let $B=B(0,1)$.
Let $\cL_0$ be an elliptic operator with constant coefficients.
For $f \in L^2(B)$ let $u \in W^{2,2}_0(B)$ be a unique solution to
\[
\cL_0 u= f\;\mbox{ in }\; B;\quad
u=0 \;\mbox{ on } \; \partial B.
\]
Then for any $\alpha>0$, we have
\[
\Abs{\set{x \in B : \abs{D^2 u(x)} > \alpha}}  \le \frac{C}{\alpha} \int_{B} \abs{f},
\]
where $C=C(d, \lambda, \Lambda)$.
\end{lemma}
\begin{proof}
The proof is a modification of Lemma~\ref{lem02-weak11}.
Since the map $T: f \mapsto D^2 u$ is a bounded linear operator on $L^2(B)$, it suffices to show that $T$ satisfies the hypothesis of Lemma~\ref{lem01-stein}.
We again take $c=2$.
Fix $\bar y \in B$, $0<r<\frac12$, and let $b \in L^2(B)$ be supported in $B(\bar y, r) \cap B$ with mean zero.
Let $u \in W^{2,2}_0(B)$ be a unique solution of
\[
\cL_0 u= b\;\mbox{ in }\; B;\quad
u=0 \;\mbox{ on } \; \partial B.
\]
For any $R\ge cr$ such that $B\setminus B(\bar y, R) \neq \emptyset$ and $\mathbf{g}=(g^{kl})_{k,l=1}^d \in C^\infty_c((B(\bar y,2R)\setminus B(\bar y,R))\cap B)$, let $v \in L^2(B)$ be a unique adjoint solution (see \cite[Lemma~2]{EM2016}) of
\[
\cL_0^* v=  \dv^2 \mathbf{g}\;\mbox{ in }\; B;\quad
v=0 \;\mbox{ on } \; \partial B,
\]
By Definition~\ref{def-adj}, we have the identity
\begin{equation}				\label{eq22.44f}
\int_{B} \tr (D^2 u \, \mathbf{g}) = \int_{B}  v b  = \int_{B(\bar y ,r)\cap B}  b \left(v-\overline{v}_{B(\bar y ,r)\cap B}\right).	
\end{equation}
Since $\cL_0^*$ is an elliptic operator with constant coefficients, $v$ is a classical solution; see \cite{EM2016}.
Since $\mathbf{g}= 0$ in $B(\bar y, R) \cap B$ and $r\le R/2$, the standard interior and boundary estimates yield
\[
\norm{Dv}_{L^\infty(B(\bar y, r) \cap B)}
\lesssim R^{-1-\frac{d}2} \norm{v}_{L^2(B(\bar y, R)\cap B)} \lesssim R^{-1-\frac{d}2} \norm{v}_{L^2(B)} \lesssim R^{-1-\frac{d}2} \norm{\mathbf{g}}_{L^2(B)},
\]
where we used \cite[Lemma~2]{EM2016} in the last step.
Therefore, we have
\begin{align*}
\Abs{\int_{(B(\bar y,2R)\setminus B(\bar y,R))\cap B}  \tr (D^2 u \, \mathbf{g}) }
&\lesssim r \norm{b}_{L^1(B(\bar y, r) \cap B)} \norm{Dv}_{L^\infty(B(\bar y, r)\cap B)} \\
&\lesssim r  R^{-1-\frac{d}2}\,\norm{b}_{L^1(B(\bar y, r)\cap B)}  \norm{\mathbf g}_{L^2((B(\bar y,2R)\setminus B(\bar y,R))\cap B)}.
\end{align*}
The rest of proof is identical to that of Lemma~\ref{lem02-weak11} and omitted.
\end{proof}

For $\bar x \in B_3$ and $0<r<\frac13$, we decompose $u=v+w$, where $w \in W^{2,2}_0(B(\bar x, r))$ is a unique solution of the problem
\[
\cL_0 w = -\tr ((A- \bar A) D^2 u) + g-\bar g \;\mbox{ in }\;B(\bar x, r);\quad
w=0 \;\mbox{ on }\;\partial B(\bar x, r).
\]
By Lemma~\ref{lem-weak11-nd} with scaling, we have for any $\alpha>0$,
\[
\abs{\set{x\in B(\bar x, r): \abs{D^2w(x)} > \alpha}} \lesssim \frac{1}{\alpha}\left(\norm{D^2 u}_{L^\infty(B(\bar x, r))} \int_{B(\bar x, r)} \abs{A-\bar A} +  \int_{B(\bar x, r)}  \abs{g -\bar g}\right).
\]
Therefore, for any $0<p<1$, we have
\[
\left(\fint_{B(\bar x, r)} \abs{D^2w}^p \right)^{\frac1p} \lesssim \omega_A(r) \,\norm{D^2 u}_{L^\infty(B(\bar x, r))} +  \omega_g(r).
\]
Since $v=u-w$ satisfies $\cL_0 v =\bar g$ in $B(\bar x, r)$,
we observe that for any $\mathbf q\in \mathrm{Sym}(d)$, the set of all $d\times d$ symmetric matrices, we have
\[
\cL_0 (D^2 v-\mathbf q) =0 \;\mbox{ in }\;B(\bar x, r).
\]
Since $\cL_0$ is an operator with constant coefficients, similar to \eqref{int-reg}, we have
\[
\norm{D^3 v}_{L^\infty(B(\bar x, \frac12 r))} \le  C_0r^{-1} \left(\fint_{B(\bar x, r)} \abs{D^2 v - \mathbf q}^p \right)^{\frac1p}
\]
and thus, similar to \eqref{eq15.50b}, we obtain (recall $0<\kappa<\frac12$)
\[
\left(\fint_{B(\bar x, \kappa r)} \abs{D^2 v - \overline{D^2 v}_{B(\bar x, \kappa r)}}^p \right)^{\frac1p} \le 2\kappa r \norm{D^3 v}_{L^\infty(B(\bar x, \frac12 r))} \le 2C_0 \kappa \left(\fint_{B(\bar x, r)} \abs{D^2 v -\mathbf q}^p \right)^{\frac1p}.
\]
If we set
\[
\phi(\bar x,r):=\inf_{\mathbf q\in \mathrm{Sym}(d)}\left( \fint_{B(\bar x,r)} \abs{D^2u -\mathbf q}^p \right)^{\frac1p},
\]
then by the same argument that led to \eqref{eq22.25f}, we get
\[
\phi(\bar x, \kappa^j r) \le 2^{-j} \phi(\bar x, r) +C \norm{D^2 u}_{L^\infty(B(\bar x, r))}\,\tilde \omega_A(\kappa^{j} r) + C \tilde \omega_g(\kappa^{j} r).
\]
Now, by repeating the same line of proof of Theorem~\ref{thm-main-d}, we reach the following estimate:
For $x, y \in B(0,1)$ such that $\abs{x-y} <\frac12$, we have
\begin{multline}				\label{eq14.00nd}
\abs{D^2 u(x)-D^2u(y)} \lesssim \norm{D^2 u}_{L^1(B_4)} \,\abs{x-y}^\beta\\
+\left(\norm{D^2 u}_{L^1(B_4)} + \int_0^{1} \frac{\tilde\omega_g(t)}t \,dt\right) \int_0^{\abs{x-y}} \frac{\tilde \omega_A(t)}t \,dt + \int_0^{\abs{x-y}} \frac{\tilde \omega_g(t)}t \,dt.
\end{multline}
The theorem is proved.

\subsection{Proof of Theorem~\ref{thm-main-adj}}
The proof of theorem is similar to that of Theorem~\ref{thm-main-nd}.
We begin with a lemma that is an adjoint version of Lemma~\ref{lem-weak11-nd}.
\begin{lemma}			\label{lem-weak11-adj}
Let $B=B(0,1)$.
Let $\cL_0$ be an elliptic operator with constant coefficients.
For $\mathbf{f}=(f^{kl})_{k,l=1}^d \in L^2(B)$, let $u \in L^2(B)$ be a unique solution to the adjoint problem
\[
\cL_0^* u= \dv^2\mathbf{f}\;\mbox{ in }\; B;\quad
u=\frac{ \mathbf{f} \nu \cdot \nu}{A_0 \nu \cdot  \nu} \;\mbox{ on } \; \partial B.
\]
Then for any $\alpha>0$, we have
\[
\Abs{\set{x \in B : \abs{u(x)} > \alpha}}  \le \frac{C}{\alpha} \int_{B} \abs{\mathbf{f}},
\]
where $C=C(d, \lambda, \Lambda)$.
\end{lemma}
\begin{proof}
By \cite[Lemma~2]{EM2016}, the map $T: \mathbf{f} \mapsto u$ is a bounded linear operator on $L^2(B)$.
We show that $T$ satisfies the hypothesis of Lemma~\ref{lem01-stein}.
As before, we take $c=2$.
For $\bar y \in B$ and $0<r<\frac12$, let $\mathbf{b}=(b^{kl})_{k,l=1}^d$ be a  matrix of $L^2(B)$ functions supported in $B(\bar y, r) \cap B$ with mean zero.
By \cite[Lemma~2]{EM2016}, there exists a unique adjoint solution  $u \in L^2(B)$ of the problem
\[
\cL_0^* u= \dv^2 \mathbf{b}\;\mbox{ in }\; B;\quad
u=\frac{ \mathbf{b} \nu \cdot \nu}{A_0 \nu \cdot  \nu}\;\mbox{ on } \; \partial B.
\]
For any $R\ge cr$ such that $B\setminus B(\bar y, R) \neq \emptyset$ and $g \in C^\infty_c((B(\bar y,2R)\setminus B(\bar y,R))\cap B)$, let  $v \in W^{2,2}_0(B)$ be a unique solution of
\[
\cL_0 v=  g\;\mbox{ in }\; B;\quad
v=0 \;\mbox{ on } \; \partial B,
\]
By Definition~\ref{def-adj}, we have the identity
\begin{equation}				\label{eq0911}
\int_{B} u \, g = \int_{B}  \tr(\mathbf b D^2v)= \int_{B(\bar y ,r)\cap B} \tr (\mathbf{b} (D^2 v-\mathbf q)),\quad\mbox{where }\; \mathbf q=\overline{D^2 v}_{B(\bar y ,r)\cap B}.
\end{equation}
Since $g = 0$ in $B(\bar y, R) \cap B$, $r\le R/c=R/2$, and $\cL_0$ has constant coefficients, the standard interior and boundary estimates and Calder\'on-Zygmund estimate yield
\[
\norm{D^3 v}_{L^\infty(B(\bar y, r)\cap B)} \lesssim R^{-1-\frac{d}2} \norm{D^2 v}_{L^2(B(\bar y, R)\cap B)} \lesssim R^{-1-\frac{d}2} \norm{D^2 v}_{L^2(B)} \lesssim R^{-1-\frac{d}2} \norm{g}_{L^2(B)}.
\]
Therefore, we have
\[
\Abs{\int_{(B(\bar y,2R)\setminus B(\bar y,R))\cap B} u\, g\,}
\lesssim r  R^{-1-\frac{d}2}\,\norm{\mathbf b}_{L^1(B(\bar y,r) \cap B)}\, \norm{g}_{L^2((B(\bar y,2R)\setminus B(\bar y,R))\cap B)}.
\]
The rest of proof is identical to that of Lemma~\ref{lem02-weak11} and omitted.
\end{proof}

For $\bar x \in B_3$ and $0<r<\frac13$, we decompose $u=v+w$, where $w \in L^2(B(\bar x, r))$ is a unique solution of the problem (see \cite[Lemma~2]{EM2016})
\[
\left\{
\begin{aligned}
\cL_0^* w &= -\dv^2 ((A- \bar A) u) + \dv^2 (\mathbf{g}-\bar{\mathbf g}) \;\mbox{ in }\;B(\bar x, r),\\
w &=\frac{\left(\mathbf{g}-\bar{\mathbf g}-(A-\bar A)u\right) \nu \cdot \nu}{\bar A \nu \cdot  \nu}\;\mbox{ on }\;\partial B(\bar x, r).
\end{aligned}
\right.
\]
By Lemma~\ref{lem-weak11-adj} with scaling, we have
\[
\abs{\set{x\in B(\bar x, r): \abs{w(x)} > \alpha}} \lesssim \frac{1}{\alpha}\left(\norm{u}_{L^\infty(B(\bar x, r))} \int_{B(\bar x, r)} \abs{A-\bar A} +  \int_{B(\bar x, r)}  \abs{\mathbf g -\bar{\mathbf g}}\,\right).
\]
Therefore, for any $0<p<1$, we have
\[
\left(\fint_{B(\bar x, r)} \abs{w}^p \right)^{\frac1p} \lesssim \omega_A(r) \,\norm{u}_{L^\infty(B(\bar x, r))} +  \omega_g(r).
\]
Note that $v$ is a local adjoint solution of  $\cL_0^* v = 0$ in $B(\bar x, r)$ and so is $v-q$ for any constant $q \in \bR$.
Since $\cL_0^*$ is an operator with constant coefficients, $v-q$ is a classical solution; see \cite{EM2016}.
Therefore, by the standard interior estimate, we have
\[
\norm{D v}_{L^\infty(B(\bar x, \frac12 r))} \le C_0 r^{-1} \left(\fint_{B(\bar x, r)} \abs{v-q}^p \right)^{\frac1p}
\]
and thus, similar to \eqref{eq15.50b}, we obtain
\[
\left(\fint_{B(\bar x, \kappa r)} \abs{v - \overline{v}_{B(\bar x, \kappa r)}}^p \right)^{\frac1p} \le 2\kappa r \norm{D v}_{L^\infty(B(\bar x, \frac12 r))} \le 2C_0 \kappa \left(\fint_{B(\bar x, r)} \abs{v - q}^p \right)^{\frac1p}.
\]
If we set
\[
\phi(\bar x,r):=\inf_{q\in \bR}\left( \fint_{B(\bar x,r)} \abs{u - q}^p \,\right)^{\frac1p},
\]
then by the same argument that led to \eqref{eq22.25f}, we get
\[
\phi(\bar x, \kappa^j r) \le 2^{-j} \phi(\bar x, r) +C \norm{u}_{L^\infty(B(\bar x, r))}\,\tilde \omega_A(\kappa^{j} r) + C \tilde \omega_g(\kappa^{j} r).
\]
Now, by repeating the same line of proof of Theorem~\ref{thm-main-d}, we reach the following estimate:
For $x, y \in B(0,1)$ such that $\abs{x-y} <\frac12$, we have
\begin{multline}				\label{eq14.00adj}
\abs{u(x)-u(y)} \lesssim \norm{u}_{L^1(B_4)} \,\abs{x-y}^\beta\\
+\left(\norm{u}_{L^1(B_4)} + \int_0^{1} \frac{\tilde\omega_g(t)}t \,dt\right) \int_0^{\abs{x-y}} \frac{\tilde \omega_A(t)}t \,dt + \int_0^{\abs{x-y}} \frac{\tilde \omega_g(t)}t \,dt. \qedhere
\end{multline}
The theorem is proved.

\section{Weak type-(1,1) estimates}				\label{sec3}
In this section, we present a condition for the coefficients $A$ to guarantee local weak type-(1,1) estimates that have been established for constants coefficients operators.
It turns out that it is sufficient that the mean oscillation $\omega_A$ satisfies the condition
\begin{equation}					\label{eq15.08m}
\omega_A(r) \lesssim ( \ln r )^{-2}, \quad\forall r\in(0,\tfrac12).
\end{equation}
In \cite{Es94}, Escauriaza constructed an elliptic operator with $\omega_A(r) \lesssim ( \ln r )^{-1}$ such that a weak type-$(1,1)$ estimate does not hold with respect to the Lebesgue measure. By modifying the counterexample in \cite{Es94}, it is easy to find elliptic operators such that
\[
\omega_A(r) \lesssim ( \ln(-\ln r) )^{-1}( \ln r )^{-1}
\quad \text{or}\quad
\omega_A(r) \lesssim ( \ln\ln(-\ln r) )^{-1}( \ln(-\ln r) )^{-1}( \ln r )^{-1}.
\]
But none of them are Dini functions, so there is a gap between \eqref{eq15.08m} and these counterexamples.

After we finished writing this paper, Luis Escauriaza \cite{E16} kindly informed us that for non-divergence form elliptic equations, a weak type-$(1,1)$ holds when the coefficients are Dini continuous. However, such result has not been explicitly written in the literature.
In a subsequent paper, we will further study weak type-$(1,1)$ estimates for non-divergence form equations with Dini mean oscillation coefficients.

Unlike Theorems \ref{thm-main-d}, \ref{thm-main-nd}, and \ref{thm-main-adj}, we do not claim that the theorems in this section can be readily extended to systems, except for Theorem \ref{thm-weak11-d}.
The reason is that we use the $W^{1,p}$ solvability for elliptic equations in divergence form with VMO (in fact, Dini mean oscillation) coefficients, which is also available for elliptic systems satisfying the strong ellipticity condition.
See, for instance, \cite[Sec. 11.1]{DK11}. However, for elliptic systems in non-divergence form, the solvability only holds for $\mathcal L-\lambda I$ with a large $\lambda$.

\begin{theorem}				\label{thm-weak11-d}
Let $B=B(0,1)$.
Suppose the coefficients $A(x)$ of the divergent operator $L$ satisfies the condition \eqref{eq15.08m}.
For $\vec f \in L^2(B)$ supported on $B(0,\frac13)$, let $u \in W^{1,2}_0(B)$ be a unique weak solution to
\[
L u= \dv \vec f\;\mbox{ in }\; B;\quad
u=0 \;\mbox{ on } \; \partial B.
\]
Then for any $\alpha>0$, we have
\[
\Abs{\set{x \in B : \abs{Du(x)} > \alpha}}  \le \frac{C}{\alpha} \int_{B} \abs{\vec f},
\]
where $C=C(d, \lambda, \Lambda, \omega_A)$.
\end{theorem}
\begin{proof}
We modify the proof of Lemma~\ref{lem02-weak11}.
Without loss of generality, we may assume that $\|f\|_{L^1(B)}=1$.
Since $f$ is supported on $B(0,\frac13)$, by virtue of the Calder\'on-Zygmund decomposition argument in Stein \cite{Stein93}, for sufficiently large $\alpha$, we only need to check the hypothesis of Lemma \ref{lem01-stein} for $\bar y\in B(0,\frac12)$ and $0<r<\frac14$.
We set $c=8$ and solve for $u$ as in the proof of Lemma~\ref{lem02-weak11}.
We replicate the same proof there with $L_0$ replaced by $L$ up to \eqref{14.23m}.
Since $L^*$ is not a constant coefficients operator, we cannot get \eqref{14.18m}.
Instead, by a similar argument that led to \eqref{eq14.00}, we get
\[
\abs{Dv(x)-Dv(y)}  \lesssim\left(\left(\frac{\abs{x-y}}{R}\right)^\beta + \int_0^{\abs{x-y}} \frac{\tilde\omega_A(t)}t \,dt \right) R^{-\frac{d}2} \norm{Dv}_{L^2(B(\bar y, \frac14 R))}
\]
for $x, y \in B(\bar y, \frac18 R)$ since $L^* v =0$ in $B(\bar y, \frac14 R) \subset B$.
Since $r\le R/c=R/8$, we thus have
\begin{align}
					\nonumber
\norm{Dv -\overline{Dv}_{B(\bar y, r)}}_{L^\infty(B(\bar y, r))} &\lesssim R^{-\frac{d}2} \norm{Dv}_{L^2(B)} \left( r^\beta R^{-\beta} + \int_0^{r} \frac{\tilde \omega_A(t)}t\,dt \right)\\
					\label{eq15.51w}
& \lesssim R^{-\frac{d}2}  \norm{Dv}_{L^2(B)}\, \left(r^\beta R^{-\beta} +\{\ln (4/r)\}^{-1}\right),
\end{align}
where we used the following lemma.
\begin{lemma}				\label{lem04-omegatilde}
Let $\omega$ be a nonnegative function such that $\omega(t) \le  \bigl(\ln \frac{t} 4\bigr)^{-2}$ for $0<t\le 1$, and $\tilde \omega$ be given as in \eqref{eq10.54}.
Then
\[
\int_0^r \frac{\tilde \omega(t)}t \,dt \lesssim  \left(\ln \frac4r \right)^{-1},\quad \forall r\in (0, 1].
\]
\end{lemma}
\begin{proof}
By definition of $\tilde \omega(t)$, we have
\[
\tilde\omega(t) \le \tilde \omega_1(t)+ \tilde \omega_2(t)
:= \sum_{j=1}^\infty \frac{[\kappa^{-j}t  \le 1]}{ 2^j (\ln (\kappa^{-j}t/4))^2}+ \frac1{(\ln 4)^2} \sum_{j=1}^\infty 2^{-j}[\kappa^{-j}t >1].
\]
We claim that
\begin{equation}					\label{eq15.37w}
\tilde \omega(t) \lesssim \left(\ln \tfrac{t}4 \right)^{-2},\quad \forall t\in (0, 1].
\end{equation}
Indeed, for $0<t <\frac12$, let $N$ be the largest integer satisfying $\kappa^{-N}t \le 1$.
We set $\gamma := \ln \kappa/ \ln \frac{t}4>0$.
Note that $\gamma N <1$ and we have
\begin{align}
					\nonumber
\left(\ln \tfrac{t}4\right)^2 \,\tilde \omega_1(t) =\sum_{j=1}^N \frac{1}{2^j (1-\gamma j)^2} &\le \sum_{j=1}^l \frac{1}{2^j (1-\gamma j)^2}+ \sum_{j=l+1}^N \frac{1}{2^j (1-\gamma N)^2} \\
					\label{eq15.30w}
& \le \frac{1}{(1-\gamma l)^2} + \frac{(\ln \frac{t}4)^2}{2^{l} (\ln 4)^2},
\end{align}
where $1<l<N$ is any integer.
If we take the integer $l$ so that  $l \sim 1/2\gamma$, then it is easy to verify that \eqref{eq15.30w} is bounded by an absolute number depending only on $\kappa$.
On the other hand, It is easy to see that $\tilde \omega_2(t) \lesssim t^\beta$, where $\beta=-\ln 2 /\ln \kappa>0$.
Therefore, we have proved  \eqref{eq15.37w} and thus the lemma.
\end{proof}

Therefore, instead of \eqref{14.18m}, we have by \eqref{eq15.51w} that
\begin{multline*}
\Abs{\int_{(B(\bar y,2R)\setminus B(\bar y,R))\cap B} Du \cdot \vec g\,}\\
\lesssim  \norm{\vec b}_{L^1(B(\bar y, r))}\left(r^\beta R^{-\beta} +\{\ln (4/r)\}^{-1}\right) R^{-\frac{d}2}  \norm{\vec g}_{L^2((B(\bar y,2R)\setminus B(\bar y,R))\cap B)}.
\end{multline*}
and thus, similar to \eqref{eq4.37}, we get
\[
\norm{Du}_{L^1((B(\bar y,2R)\setminus B(\bar y,R))\cap B)} \lesssim \left(r^\beta R^{-\beta} +\{\ln (4/r)\}^{-1}\right)\,\norm{\vec b}_{L^1(B(\bar y, r))}.
\]
Recall that $N$ is the smallest positive integer such that $B \subset B(\bar y, 2^N cr)$. Therefore, we have $N \sim \ln(1/r)$, and similar to \eqref{14.06m}, we obtain
\[
\int_{B \setminus B(\bar y, cr)} \abs{D u} \lesssim \sum_{k=1}^N \left(2^{-\beta k} +\{\ln (4/r)\}^{-1}\right)\ \norm{\vec b}_{L^1(B(\bar y, r))} \sim \int_{B(\bar y, r)} \abs{\vec b},
\]
and thus we are done.
\end{proof}

\begin{theorem}				\label{thm-weak11-nd}
Let $B=B(0,1)$.
Suppose the coefficients $A(x)$ of the non-divergent operator $\cL$ satisfies the condition \eqref{eq15.08m}.
For $f \in L^2(B)$ supported on $B(0,\frac13)$, let $u \in W^{2,2}_0(B)$ be a unique solution to
\[
\cL u= f\;\mbox{ in }\; B;\quad
u=0 \;\mbox{ on } \; \partial B.
\]
Then for any $\alpha>0$, we have
\[
\Abs{\set{x \in B : \abs{D^2u(x)} > \alpha}}  \le \frac{C}{\alpha} \int_{B} \abs{f},
\]
where $C=C(d, \lambda, \Lambda, \omega_A)$.
\end{theorem}
\begin{proof}
As explained in the proof of Theorem~\ref{thm-weak11-d}, we only need to check the hypothesis of Lemma~\ref{lem01-stein} for $\bar y\in B(0,\frac12)$ and $0<r<\frac14$. We set $c=8$ and follow the same line of  proof of Lemma~\ref{lem-weak11-nd} up to \eqref{eq22.44f} with $\cL$ in place of $\cL_0$.
Since $\cL^* v =0$ in $B(\bar y, \frac14 R)\subset B$, similar to \eqref{eq14.00adj}, we have
\[
\abs{v(x)-v(y)}  \lesssim\left(\left(\frac{\abs{x-y}}{R}\right)^\beta + \int_0^{\abs{x-y}} \frac{\tilde\omega_A(t)}t \,dt \right) R^{-\frac{d}2} \norm{v}_{L^2(B(\bar y, \frac14 R))}
\]
for $x, y \in B(\bar y, \frac18 R)$.
Then, similar to \eqref{eq15.51w}, we get
\[
\norm{v -\overline{v}_{B(\bar y, r)}}_{L^\infty(B(\bar y, r))}
\lesssim  R^{-\frac{d}2} \left(r^\beta R^{-\beta} +\{\ln (4/r)\}^{-1}\right) \norm{\mathbf g}_{L^2(B)},
\]
where we used \cite[Lemma~2]{EM2016} to bound $\norm{v}_{L^2(B)}$ by $\norm{\mathbf g}_{L^2(B)}$.
Therefore, by \eqref{eq22.44f} and the above inequality, we get
\begin{multline*}
\Abs{\int_{(B(\bar y,2R)\setminus B(\bar y,R))\cap B}  \tr (D^2 u \, \mathbf{g}) }\\
\lesssim R^{-\frac{d}2} \left(r^\beta R^{-\beta} +\{\ln (4/r)\}^{-1}\right)\norm{b}_{L^1(B(\bar y, r))}\norm{\mathbf g}_{L^2((B(\bar y,2R)\setminus B(\bar y,R))\cap B)},
\end{multline*}
and thus, similar to \eqref{14.06m}, we get
\[
\norm{D^2u}_{L^1((B(\bar y,2R)\setminus B(\bar y,R))\cap B)} \lesssim \left(r^\beta R^{-\beta} +\{\ln (2/r)\}^{-1}\right)\,\norm{b}_{L^1(B(\bar y, r))}.
\]
The rest of proof is identical to that of Theorem~\ref{thm-weak11-d}.
\end{proof}

\begin{theorem}				\label{thm-weak11-adj}
Let $B=B(0,1)$.
Suppose the coefficients $A(x)$ of the non-divergent operator $\cL$ satisfies the condition \eqref{eq15.08m}.
For $\mathbf{f}=(f^{kl})_{k,l=1}^d \in L^2(B)$ supported on $B(0,\frac13)$, let $u \in L^2(B)$ be a unique solution to the adjoint problem
\[
\cL^* u= \dv^2\mathbf{f}\;\mbox{ in }\; B;\quad
u=0 \;\mbox{ on } \; \partial B.
\]
Then for any $\alpha>0$, we have
\[
\Abs{\set{x \in B : \abs{u(x)} > \alpha}}  \le \frac{C}{\alpha} \int_{B} \abs{\mathbf f},
\]
where $C=C(d, \lambda, \Lambda, \omega_A)$.
\end{theorem}
\begin{proof}
As explained in the proof of Theorem~\ref{thm-weak11-d}, we only need to check the hypothesis of Lemma~\ref{lem01-stein} for $\bar y\in B(0,\frac12)$ and $0<r<\frac14$. We set $c=8$ and follow the same line of proof of Lemma~\ref{lem-weak11-adj} up to \eqref{eq0911} with $\cL$ and $\cL^*$ in place of $\cL_0$ and $\cL_0^*$.
Since $\cL^* v =0$ in $B(\bar y, \frac14 R)\subset B$, similar to \eqref{eq14.00nd}, we have
\[
\abs{D^2v(x)-D^2v(y)}  \lesssim\left(\left(\frac{\abs{x-y}}{R}\right)^\beta + \int_0^{\abs{x-y}} \frac{\tilde\omega_A(t)}t \,dt \right) R^{-\frac{d}2} \norm{D^2v}_{L^2(B(\bar y,\frac14 R))}
\]
for $x, y \in B(\bar y, \frac18 R)$.
Then, similar to \eqref{eq15.51w}, we obtain
\[
\norm{D^2v -\overline{D^2v}_{B(\bar y, r)}}_{L^\infty(B(\bar y, r))}
\lesssim  R^{-\frac{d}2} \left(r^\beta R^{-\beta} +\{\ln (4/r)\}^{-1}\right) \norm{g}_{L^2(B)},
\]
where we used the estimate $\norm{D^2 v}_{L^2(B)} \lesssim \norm{g}_{L^2(B)}$.
Therefore, by \eqref{eq0911} and the above inequality, we get
\[
\Abs{\int_{(B(\bar y,2R)\setminus B(\bar y,R))\cap B} u\, g\,}
\lesssim   R^{-\frac{d}2} \left(r^\beta R^{-\beta} +\{\ln (4/r)\}^{-1}\right) \norm{b}_{L^1(B(\bar y, r))}
\norm{g}_{L^2((B(\bar y,2R)\setminus B(\bar y,R))\cap B)}
\]
and thus, similar to \eqref{14.06m}, we get
\[
\norm{u}_{L^1((B(\bar y,2R)\setminus B(\bar y,R))\cap B)} \lesssim \left(r^\beta R^{-\beta} +\{\ln (2/r)\}^{-1}\right)\,\norm{b}_{L^1(B(\bar y, r))}.
\]
The rest of proof is identical to that of Theorem~\ref{thm-weak11-d}.
\end{proof}

\section{An application}				\label{sec4}
In \cite{Br08,An09}, Brezis answered a question raised by Serrin \cite{Se64} by proving that any $W^{1,1}$-weak solution to divergence form equations is in $W^{1,p}$ for any $p\in (1,\infty)$ provided that the coefficients are Dini continuous.
Recently, Escauriaza and Montaner \cite{EM2016} obtained a similar result for non-divergence form equations under the same Dini condition.
The proofs in these two papers are based on a duality argument and use the boundedness of the gradient of solutions (or solutions themselves) to the adjoint equations.
See also \cite{JMV09} and \cite{EM2016} for the corresponding counterexamples.

It follows from Theorems \ref{thm-main-d} and \ref{thm-main-adj} that for the boundedness of the gradient of solutions (or solutions themselves) to the adjoint equations, the Dini continuity condition can be replaced by the Dini mean oscillation condition \eqref{eq:dcmo}.
As a consequence, we deduce the following corollaries, which improve the aforementioned results in \cite{Br08, An09, EM2016}.

\begin{corollary}
                    \label{cor1}
Suppose the coefficients $A(x)$ of the divergent operator $L$ have Dini mean oscillation; i.e., $A$ satisfies \eqref{eq:dcmo}.
Let $u \in W^{1,1}(B_4)$ be a weak solution of
\[
Lu= \dv (A(x) \nabla u)=\dv \vec g\;\mbox{ in } \; B_4=B(0,4),
\]
where $\vec g\in L^p(B_4)$ for some $p\in (1,\infty)$.
Then, we have $u\in W^{1,p}(B_1)$ and
\[
\norm{u}_{W^{1,p}(B_1)}\le C \norm{u}_{W^{1,1}(B_4)}+C \norm{\vec g}_{L^p(B_4)},
\]
where $C$ depends only on $d$, $p$, $\lambda$, $\Lambda$, the uniform
modulus of continuity of the coefficients.
\end{corollary}

\begin{corollary}
                        \label{cor2}
Suppose the coefficients $A(x)$ of the non-divergent operator $\cL$ have Dini mean oscillation; i.e., $A$ satisfies  \eqref{eq:dcmo}.
Let $u \in W^{2,1}(B_4)$ be a strong solution of
\[
\cL u= \tr (A(x) D^2 u)= g\;\mbox{ in } \; B_4=B(0,4),
\]
where $g\in L^p(B_4)$ for some $p\in (1,\infty)$.
Then, we have $u\in W^{2,p}(B_1)$ and
\[
\norm{u}_{W^{2,p}(B_1)}\le C \norm{u}_{W^{2,1}(B_4)}+C \norm{g}_{L^p(B_4)},
\]
where $C$ depends only on $d$, $p$, $\lambda$, $\Lambda$, the uniform
modulus of continuity of the coefficients.
\end{corollary}

We note that regarding the dependence of $C$ in Corollaries \ref{cor1} and \ref{cor2}, the uniform modulus of continuity of the coefficients can be replaced by the modulus of continuity of the coefficients in the $L^1$-mean sense.

\section*{Acknowledgement}
The authors would like to thank the referees for reading of the manuscript and many useful comments.


\end{document}